\documentclass[a4paper,10pt]{article}

\RequirePackage[OT1]{fontenc}
\RequirePackage{amsthm,amsmath}
\RequirePackage[numbers]{natbib}
\RequirePackage[colorlinks,citecolor=blue,urlcolor=blue]{hyperref}

\usepackage{color}
\usepackage{amsmath,amssymb,amsthm}
\usepackage{latexsym}
\usepackage{graphicx}
\usepackage{epsfig}
\usepackage{float,enumerate}

\newtheorem{theo}{Theorem}

\newtheorem{prop}{Proposition}
\newtheorem{lem}{Lemma}
\newtheorem{definition}{Definition}

\newtheorem{hypo}{Cross Model}

\newcommand{\set}[1]{\left\{#1\right\}}

\newcommand{\eps}{\varepsilon}
\newcommand{\lf}{\lfloor}
\newcommand{\rf}{\rfloor}

\newcommand{\Tore}{K,d}
\newcommand{\Bande}{\mathbb{Z}_K}
\newcommand{\Zero}{\mathbf{0}}

\renewcommand{\P}{\mathbf{P}}
\newcommand{\E}{\mathbf{E}}
\newcommand{\N}{\mathbb{N}}
\newcommand{\Z}{\mathbb{Z}}

\newcommand{\un}{{\mathchoice {\rm 1\mskip-4mu l} {\rm 1\mskip-4mu l} {\rm 1\mskip-4.5mu l} {\rm 1\mskip-5mu l}}}
\newcommand{\defeq}{\ensuremath{\overset{\hbox{\tiny{def}}}{=}}}
\newcommand{\infi}{\un_{\{\Zero\leftrightarrow \infty\}}}
\newcommand{\infin}{\un_{\{\Zero\leftrightarrow (n,0)\leftrightarrow \infty\}}}
\newcommand{\IntG}{[\![}
\newcommand{\IntD}{]\!]}

\title{Distances in the highly supercritical percolation cluster}
\author{Anne-Laure Basdevant, Nathana\"{e}l Enriquez, Lucas Gerin}

\begin{document}
\maketitle
\begin{abstract}

On the supercritical percolation cluster with parameter $p$, the distances between two distant points of the axis are asymptotically increased by a factor $1+\frac{1-p}{2}+o(1-p)$ with respect to the usual distance.
The proof is based on an apparently new connection with the TASEP (totally asymmetric simple exclusion process).

\end{abstract}


\section{Introduction}

First passage percolation is a model introduced in the 60's by Hammersley and Welsh \cite{HW} which asks the question of the minimal distance $D(x)$ between the origin $\Zero$ and a distant point $x$ of $\Z^2$, when edges have i.i.d. positive finite lengths.
One can prove by subadditivity arguments that in every direction such distances grow linearly: for each $y\in \Z^2\backslash\{0\}$,
$D(ny)/n$ converges almost surely to a constant $\mu(y)$. The particular value of $\mu((1,0))=: \mu$ is called the time constant. It is unknown except in the trivial case of deterministic edge lengths. We refer the reader to \cite{Kes} for an introduction on first passage percolation.

In this article we study the extreme case where edges have lengths $1$ with probability $p\in(0,1)$, $+\infty$ with probability $1-p$. Then, the distance $D(x)$ coincides with the distance between the origin and $x$ in the graph induced by bond percolation on $\Z^2$: each edge of $\Z^2$ is open with probability $p$ and closed with probability $1-p$. When $p>1/2$, it is known (see \cite{Grimmett} Chap.1 for an introduction to bond percolation) that there exists almost surely a unique infinite connected component of open edges.
We write $x\leftrightarrow y$ if $x,y$ belong to the same connected component, and  $x\leftrightarrow \infty$ if $x$ is in the infinite cluster.

G\"{a}rtner and Molchanov (\cite{Gartner} Lemma 2.8) were the first to rigorously prove that if $\Zero$ and $x$ belong to the infinite component, $D(x)$ is of order $x$.
Garet and Marchand (\cite{GM}, Th.3.2) improved this result and showed that, even if the subadditivity argument fails, the limit still holds: for each $y\in \Z^2\backslash\{0\}$, there exists a constant $\mu(y)$ such that, on the event $\set{\Zero\leftrightarrow \infty}$, we have a.s.
\begin{equation*}\lim_{\substack{n \to \infty \\\Zero\leftrightarrow (n,0)}} \frac{D(ny)}{n}=\mu(y).\end{equation*}

The aim of the present paper is to give an asymptotics of the time constant when $p$ is close to one.
One clearly has $D(n,0)\geq n$. On the other side, among the $n$ edges of the segment joining $\Zero$ to $(n,0)$, about $n(1-p)$ of them are closed but with high probability the two extremities of each such edge can be joined by a path of length three. This naive approach suggests an upper bound of $1+2(1-p)$ for the time constant. To our knowledge, the best known upper bound comes from
Corollary 6.4 of \cite{GM} and  is equal to $1+(1-p)$.
We obtain in the present paper a sharp asymptotics of the time constant when $p$ goes to one.
\begin{theo}\label{main}
On the event $\set{\Zero\leftrightarrow \infty}$, we have a.s.\footnote{
We use notations introduced in \cite{GM}: the subscript $\Zero\leftrightarrow (n,0)$ means that we only take $n$'s for which $(n,0)$ is in the infinite component. 
More precisely, if $(T_1(n),0)$ stands for the $n$-th point of the half-line $\mathbb{N}\times\set{0}$ belonging to the infinite component, then
 $\lim_{\substack{n \to \infty \\\Zero \leftrightarrow (n,0)}} \frac{D(n,0)}{n}:= \lim_{n \to \infty} \frac{D(T_1(n),0)}{T_1(n)}$.}
\begin{equation*}
\mu_p\defeq\lim_{\substack{n \to \infty \\\Zero \leftrightarrow (n,0)}} \frac{D(n,0)}{n} =1+\frac{1-p}{2}+\mathrm{o}(1-p).
\end{equation*}
\end{theo}

Our result says that the graph distance and the $L^1$ distance asymptotically differ by a factor $1+(1-p)/2$.
Note that Garet and Marchand (\cite{GM} Cor.6.4) observed that these two distances coincide in all the directions inside a cone containing the axis $\set{y=x}$. The angle of this cone is characterized by the asymptotic speed of oriented percolation of parameter $p$, studied by Durrett \cite{PercoOrientee}.

The key ingredient of the proof relies on a correspondance between the synchronous \emph{totally asymmetric simple exclusion process} (TASEP) on an interval and the graph distance on the percolation cluster inside an infinite strip.

\section{First bounds on $\mu_p$}
We denote by $\P_p$ the product measure on the set of edges of $\Z^2$ of length 1 under which each edge is open independently with probability $p$. Since $p>1/2$, we have $\P_p\{\Zero\leftrightarrow \infty\}>0$, so we can also define $\bar{\P}_p$ the probability $\P_p$ conditioned on the event $\{\Zero\leftrightarrow \infty\}$.
\begin{equation*}
\bar{\P}_p\{A\}=\frac{{\P}_p\{A\cap \{\Zero\leftrightarrow \infty\}\}}{\P_p\{\Zero\leftrightarrow \infty\}}.
\end{equation*}
When no confusion is possible, we will omit the subscript $p$.

The origin is in the infinite cluster, unless there is a path of closed edges in the dual lattice surrounding $\Zero$. In the whole paper, we take $p$ close enough to one so that this occurs with 
high probability (to fix ideas, with probability greater than $1-2(1-p)^4$).

Because of the conditioning, it is not possible to apply directly subadditive arguments to the sequence $D(n,0)$.
To overcome this difficulty, we adapt the ideas of \cite{GM} and consider the sequence of points of the axis which lie in the infinite cluster. This enables us to derive bounds on $\mu_p$.

For $n\ge 1$, let
$(T_n(k),0)$ be the $k$-th intersection of the infinite cluster and the set $\{(in,0), i\in \N\}$.
\begin{prop}\label{Prop:1ereBorne}
We have for all $n\ge 1$,
\begin{equation*}\lim_{k\rightarrow \infty} \E\left(\frac{D(T_1(k),0)}{k}\un_{\{\Zero\leftrightarrow \infty\}}\right) = \mu_p\le \E\left(\frac{D(T_n(1),0)}{n}\infi\right).\end{equation*}
\end{prop}
\begin{proof} This proposition is mainly a direct consequence of Lemma 3.1 of \cite{GM}. Indeed, this lemma states that, for all $n\ge 1$, there exists a constant $f$ such that
\begin{equation}\label{sousadd}
\lim_{k\rightarrow \infty} \frac{D(T_n(k),0)}{nk}=f\qquad \bar{\P}  \mbox{ a.s. and in } L^1(\bar{\P}).
\end{equation}
Moreover, by subadditivity, we have $nf\le \E\left(D(T_n(1),0)|\;\Zero\leftrightarrow \infty\right)$ and besides
\begin{equation*}\mu_p=\P\{\Zero\leftrightarrow \infty\}f.\end{equation*}
Combining these two facts, we get the upper bound.
For the left equality, we now use the $L^1$ convergence in  (\ref{sousadd}) with $n=1$. This gives
\begin{equation*}\lim_{k\rightarrow \infty}\E\left(\frac{D(T_1(k),0)}{k}\;\Big|\;\Zero\leftrightarrow \infty\right)=f=\frac{\mu_{p}}{\P\{\Zero\leftrightarrow \infty\}}.\end{equation*}
\end{proof}
For $p$ close to 1, the upper bound can be simplified using the following proposition.
\begin{prop}\label{bornemu} For all $\delta>1$, there exists $c_{\delta}>0$ such that for all $n\ge 10$ and $p\in (5/6,1)$, we have
\begin{equation*}
\E(D(T_n(1),0)\infi)\le \E(D(n,0)\infin)+c_{\delta}(1-p)^2n^{\delta}.
\end{equation*}
\end{prop}
To prove this proposition, we first  show two lemmas.
\begin{lem}\label{roughbound} For all $\delta>1$, there exists a constant $C_\delta>0$ such that, for $n\ge 1$ and  for $p\in (5/6,1)$, we have
\begin{equation*}
\E(D^2(n,0)\infin)\le (C_{\delta}n^{\delta})^2.
\end{equation*}
\end{lem}
\begin{proof}
We have
\begin{equation*}
\E(D^2(n,0)\infin)\le n^{2\delta}+ \sum_{i=n^{2\delta}}^{\infty}\P\{D(n,0)\ge \sqrt{i}, \Zero\leftrightarrow (n,0)\leftrightarrow \infty\}.
\end{equation*}
Fix some $q\in (0,\frac{1}{2\delta})$ such that $q+\frac{1}{2\delta}\le \frac{1}{2}$ and for
 $i\ge n^{2\delta}$, let $\Gamma_i$ be the box $[-i^{\frac{1}{2\delta}},n+i^{\frac{1}{2\delta}}]\times[- \frac{i^{q}}{6},\frac{i^{q}}{6}]$.
A self avoiding path in $\Gamma_i$ has less than $|\Gamma_i|=\frac{i^{q}}{3}(n+2i^{\frac{1}{2\delta}})\le \sqrt{i}$ steps. Thus
\begin{equation*}
\P\{D(n,0)\ge \sqrt{i}, \Zero\leftrightarrow (n,0) \leftrightarrow  \infty\}\le \P\{\Zero\nleftrightarrow {(n,0)} \mbox{ in } \Gamma_i, \Zero\leftrightarrow (n,0) \leftrightarrow \infty\}.
\end{equation*}
The event $\{\Zero\nleftrightarrow {(n,0)} \mbox{ in } \Gamma_i, \Zero\leftrightarrow (n,0) \leftrightarrow \infty\}$ implies the existence in the dual of a path of closed edges starting from the border of $\Gamma_i$, and having at least $\frac{i^{q}}{3}$ steps since it must disconnect $\Zero$ and $(n,0)$. Using that $|\partial \Gamma_i |\le 6i^{\frac{1}{2\delta}}$, we get
\begin{equation*}\P\{D(n,0)\ge \sqrt{i}, \Zero\leftrightarrow (n,0) \leftrightarrow \infty\}\le 6i^{\frac{1}{2\delta}}(3(1-p))^{\frac{i^{q}}{3}}.\end{equation*}
This yields, for $p\in (5/6,1)$,
\begin{equation*}\E(D^2(n,0)\infi)\le n^{2\delta} + 6\sum_{i=n^{2\delta}}^{\infty}  i^{\frac{1}{2\delta}}\left(\frac{1}{2}\right)^{\frac{i^{q}}{3}}\le (C_{\delta}n^{\delta})^2 .\end{equation*}
\end{proof}

\begin{lem}\label{Lem:Tn} Recall that $T_n(1)$ denotes the first point among $\set{n,2n,3n,\dots}$ which is in the infinite component. There exists $C>0$ such that for any $p\in(5/6,1)$, any $n\ge 10$ and $j\ge 2$
\begin{equation}
\P\set{T_n(1)\ge jn}\le C(1-p)^{4+\sqrt{j-2}}.\label{Majo:Tn}
\end{equation}
\end{lem}

\begin{proof}
For $j=2$, the left-hand side in \eqref{Majo:Tn}  is simply $\P\{(n,0)\text{ is disconnected }\}\leq C(1-p)^4$. For $3\leq j\leq 10$, the event $\set{T_n(1)\ge jn}$ is included in $\{(n,0)\nleftrightarrow \infty $ and $(2n,0)\nleftrightarrow \infty\}$. These two points are either disconnected by two different paths, or by the same path (which then has length $\geq 2n$). The latter case has a much smaller probability, and then $\P\{T_n(1)\ge jn\}\leq C'(1-p)^8$, and \eqref{Majo:Tn} holds.

We now do the case $j\geq 10$. For each integer $i$, let $A_i$ be the event that $(ni,0)$ is disconnected, and let
\begin{equation*}
A_{i,r}=\set{(ni,0)\text{ is disconnected by a path included in }[ni-r;ni+r]\times[-r;r]}.
\end{equation*}
One plainly has that each $A_{i,r}\subset A_i$ and that $A_{i,r},A_{j,r}$ are independent as soon as the corresponding boxes are disjoint, namely if $2r\leq n|j-i|$. For $j\geq 2$, set $J=\lf \sqrt{j}\rf$, we write
\begin{align*}
\P\{T_n(1)\ge jn\}&=\P\{A_{1},A_{2},\dots,A_{j-1}\}\leq\P\{A_{J},A_{2J},A_{(J-1)\times J}\}\\
&\leq \prod_{\ell=1}^{J-1} \P\{A_{\ell J,nJ/2}\}
+ \P\{ \exists \ell\leq J-1 ; A_J \setminus A_{\ell J,nJ/2}\}\\
&\leq (1-p)^{4(J-1)} +J\sum_{i\geq nJ}i(3(1-p))^i\\
&\leq (1-p)^{4(J-1)}+ 4nJ^2(3(1-p))^{nJ}\\
&\leq C(1-p)^{4+\sqrt{j-2}},
\end{align*}
since a path disconnecting $(n\ell J,0)$ which is not included in the box $[n\ell J -\frac{nJ}{2};n\ell J+\frac{nJ}{2}]\times[-\frac{nJ}{2};\frac{nJ}{2}]$ has at least $nJ$ edges.
\end{proof}

We are now able to prove Proposition \ref{bornemu}.
\begin{proof}[Proof of Proposition \ref{bornemu}]
We write, using Lemmas \ref{roughbound} and \ref{Lem:Tn}:
\begin{eqnarray*}
\E(D(T_n(1),0)\infi)&=&\sum_{j=1}^\infty\E(D(jn,0)\un_{\{T_n(1)=jn\}}\infi)\\
&\le & \E(D(n,0)\infin)\\
& &+\sum_{j=2}^\infty\E(D(jn,0)^2\un_{\{0\leftrightarrow (jn,0)\leftrightarrow \infty\}})^{1/2}\P\{T_n(1)=jn\}^{1/2}\\
&\le & \E(D(n,0)\infin)+\sum_{j=2}^\infty C_{\delta}(jn)^{\delta}\P\{T_n(1)\ge jn\}^{1/2}\\
&\le & \E(D(n,0)\infin)+C'_{\delta}n^{\delta}(1-p)^2 \sum_{j=2}^\infty j^{\delta}\left(\frac{1}{6}\right)^{(j-2)^{1/4}}\\
& \le &\E(D(n,0)\infin)+c_{\delta}(1-p)^2n^{\delta}.\\
\end{eqnarray*}
\end{proof}


\section{Percolation on a strip and TASEP}\label{Sec:Markov}

As a first step towards our main result, we shall study distances in percolation on an infinite strip. We will reduce this problem to the analysis of a finite particle system, this allows explicit computations from which will result the bounds in Theorem \ref{main}.

Here is the context we will deal with in the whole section.
Fix an integer $K$ and $\varepsilon\in(0,1)$.
Let $\Bande$ be the infinite strip $\Z\times \IntG -K,K\IntD$, with three kinds of edges :
\begin{itemize}
\item Vertical edges $\set{(i,j)\to(i,j+1),i\in \Z,j\in \IntG -K,K-1\IntD}$;
\item Horizontal edges $\set{(i,j)\to(i+1,j),i\in \Z,j\in \IntG -K,K\IntD}$;
\item Diagonal edges $\set{(i,j)\to(i+1,j+1)\text{ and }(i,j)\to(i+1,j-1)}$.
\end{itemize}

We now consider a random subgraph of $\Bande$ equipped with distances:
\begin{hypo}\label{Model:Tore}
\begin{itemize}
\item[\emph{(i)}] Vertical and horizontal edges have length $1$, whereas diagonal edges have length $2$.
\item[\emph{(ii)}] Diagonal and vertical edges are open.
\item[\emph{(iii)}] Each horizontal edge is open (resp. closed) independently with probability $1-\varepsilon$ (resp. $\varepsilon$).
\end{itemize}
\end{hypo}

For $i\ge 0$ and $j\in \IntG -K,K\IntD $, let $D^{\Tore}(i,j)$ be the distance between $(0,0)$ and $(i,j)$ inside $\Bande$ in the Cross Model (see an example in Fig. 1), the '$d$' stands for the addition of diagonal edges. Sometimes we also need to consider the distance between two vertices $x,y$
of $\Bande$, which will be denoted by $D^{\Tore}(x\to y)$.

Since vertical and diagonal edges are open, every point in $\Bande$ is connected in the Cross Model to $\Zero$, hence $D^{\Tore}(i,j)$ is finite for every $i,j$. We also set $\mathbf{D}^{\Tore}_i=\set{D^{\Tore}(i,j),j\in \IntG -K,K\IntD}$.
By construction, we have along each vertical edge $|D^{\Tore}(i,j)-D^{\Tore}(i,j+1)|=1$.

\begin{figure}[h!]
\label{Fig:TASEP}
\begin{center}
\includegraphics[width=70mm]{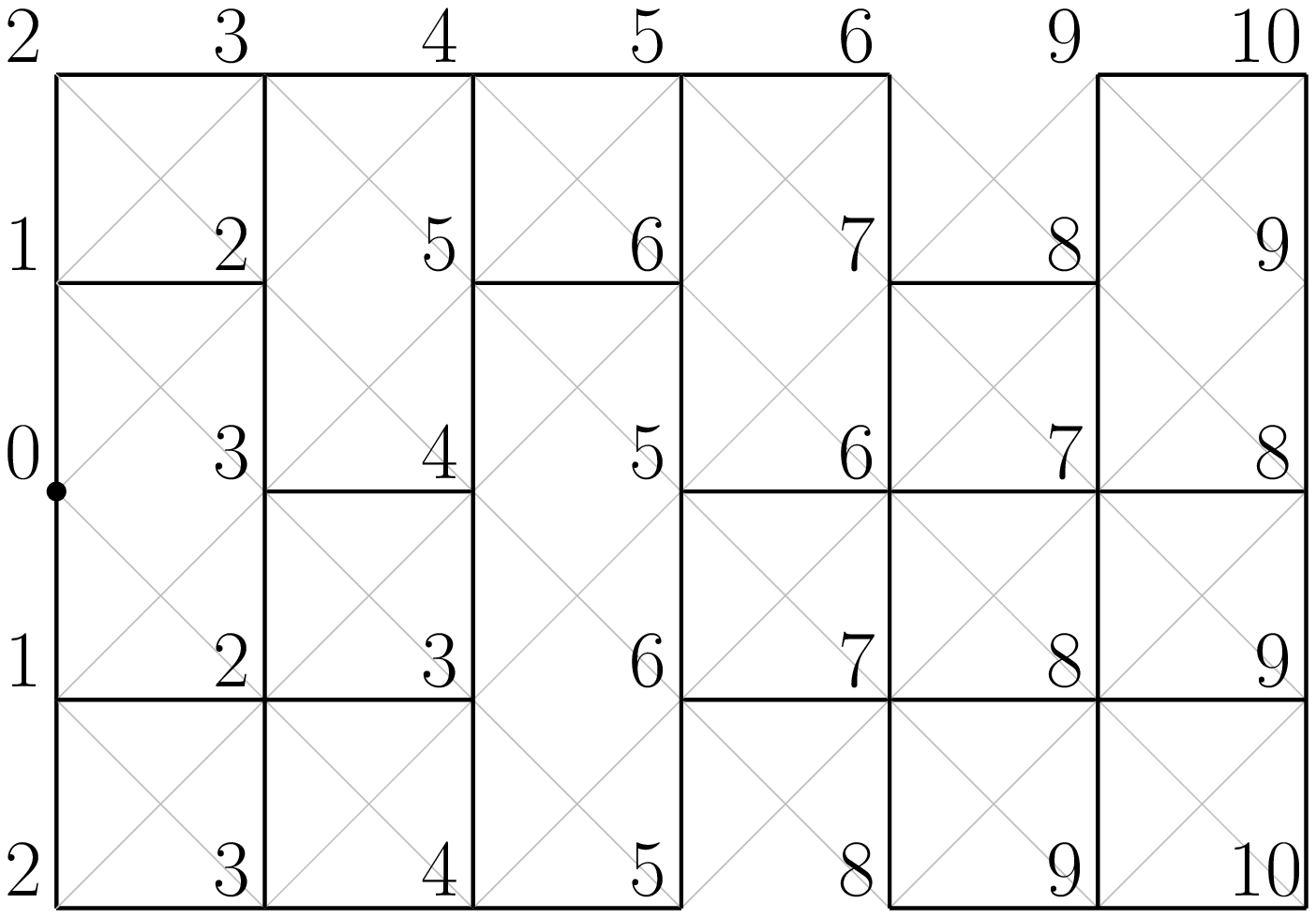}
\caption{A configuration of percolation in $\Bande$ for $K=2$ and the associated distances
$\mathbf{D}^{\Tore}$. Note the importance of diagonal edges: $\mathbf{D}^{\Tore}(3,0)=5$ instead of $7$ if there was none.}
\end{center}
\end{figure}

The main goal of this section is to estimate $D^{\Tore}(n,0)$, when $K,n$ are large. To do so, we introduce a particle system associated to the process $(\mathbf{D}^{\Tore}_i)_{i\ge 0}$. 
Let us consider the state space $\set{\bullet,\circ}^{2K}$ (identified to $\set{1,0}^{2K}$), and denote its elements in the form
$$
(y^{-K+1},y^{-K+2},\dots,y^{0},y^1,\dots,y^K).
$$
Let $(\mathbf{Y}_i)_{i\geq 0}$ be the process with values in
$\set{\bullet,\circ}^{2K}$  defined as follows :
\begin{equation*}\forall j\in \IntG -K+1,K\IntD,\qquad
Y_i^j =
\begin{cases}
&\bullet =1 \text{ if } D^{\Tore}(i,j)= D^{\Tore}(i,j-1) -1.\\
&\circ =0  \text{ if } D^{\Tore}(i,j)= D^{\Tore}(i,j-1) +1.
\end{cases}
\end{equation*}
Let say that the site $j$ is \emph{occupied} by a particle at time $i$ if $Y^j_i=\bullet$ and \emph{empty} otherwise.

\begin{figure}[h!]
\label{Fig:TASEP_avec_particules}
\begin{center}
\includegraphics[width=70mm]{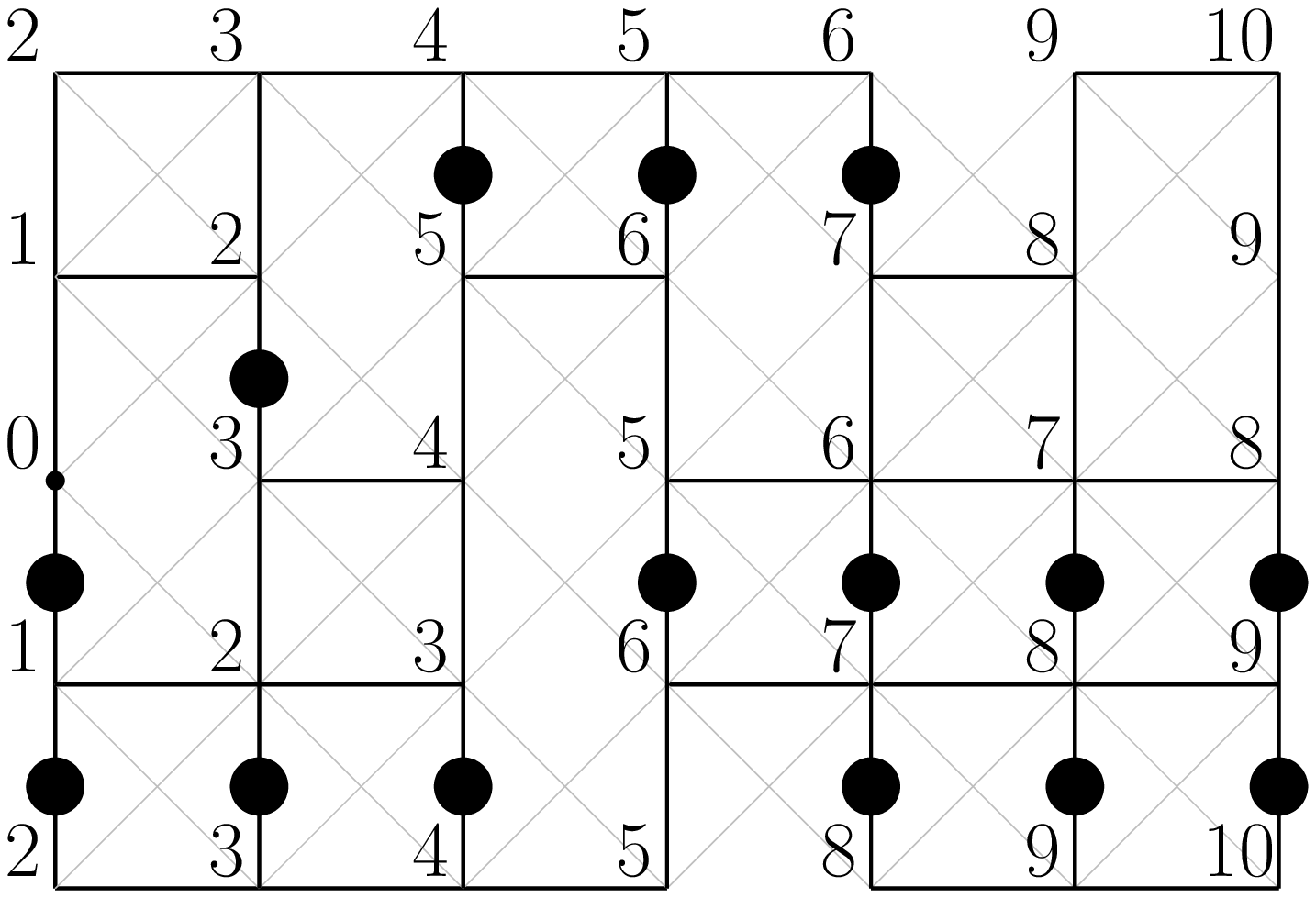}
\caption{The same configuration with particles.}
\end{center}
\end{figure}

\begin{definition}
The synchronous \emph{Totally Asymmetric Simple Exclusion Process} (TASEP) on $\IntG -K+1,K\IntD$ with jump rate $\alpha$, exit rate $\beta$ and entry rate $\gamma$ is the Markov chain with state space $\set{\bullet,\circ}^{2K}$ defined as follows:
\begin{center}
\includegraphics[width=50mm]{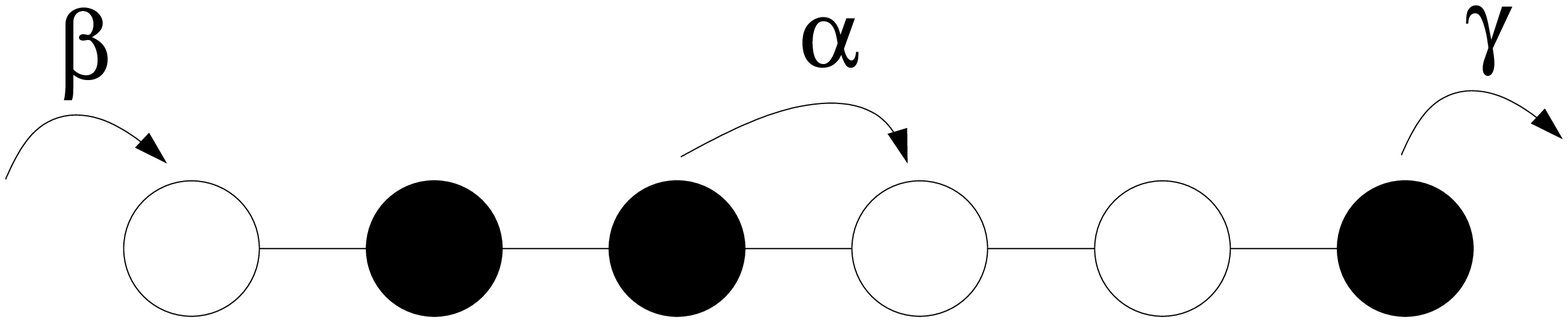}
\end{center}
\begin{itemize}
\item at time $t+1$, for each $j$, a particle at position $j=-K+1,\dots, K-1$ moves one step forward if the site $j+1$ is empty at time $t$, with probability $\alpha$ and independently from the other particles.
\item at time $t+1$, a particle enters the system at position $-K+1$ if site $-K+1$ is empty at time $t$, with probability $\beta$.
\item at time $t+1$, if there were a particle at position $K$ at time $t$, it exits the system with probability $\gamma$.
\end{itemize}
\end{definition}

\begin{prop}[The particles follow a TASEP]\label{Prop:ChaineMarkov}
The processes $(\mathbf{D}^{\Tore}_i)_{i\ge 0}$ and $(\mathbf{Y}_i)_{i\ge 0}$
are Markov chains. Moreover, $(\mathbf{Y}_i)_{i\ge 0}$ has the law of discrete time synchronous TASEP on $\IntG -K+1,K\IntD$ with jump rate $\varepsilon$ and exit and entry rate $\varepsilon$.
\end{prop}

\begin{proof}[Proof of Proposition \ref{Prop:ChaineMarkov}]
Let us note that since all the vertical edges are open, the optimal path from $\Zero$ to $(i,j)\in \Bande$, $i\ge 0$ never does a step from right to left.
Moreover, the vector $\mathbf{D}^{\Tore}_{i+1}$ depends only on $\mathbf{D}^{\Tore}_i$ and on the edges $\{(i,j)\to (i+1,j),j\in  \IntG -K,K\IntD\}$, hence it is Markov.

Let us now prove that the displacement of particles follows the rules of TASEP. We detail the case in which there is a particle on the edge $(i,j-1)\to (i,j)$ but no particle on the edge $(i,j)\to (i,j+1)$, \emph{i.e.} at time $i$ there is a particle in position $j$ and no particle in position $j+1$.

This means that if $D^{\Tore}(i,j)=\ell$ then $D^{\Tore}(i,j-1)=D^{\Tore}(i,j+1)=\ell+1$. Then, whether the horizontal edges $(i,j-1)\to (i+1,j-1)$ and $(i,j+1)\to (i+1,j+1)$ are open or closed, we have $D^{\Tore}(i+1,j-1)= D^{\Tore}(i+1,j+1)=\ell +2$ (this is because the two diagonal edges starting from $(i,j)$ are open), see the left part of the figure below:

\vspace{5mm}
\begin{center}
\includegraphics[width=65mm]{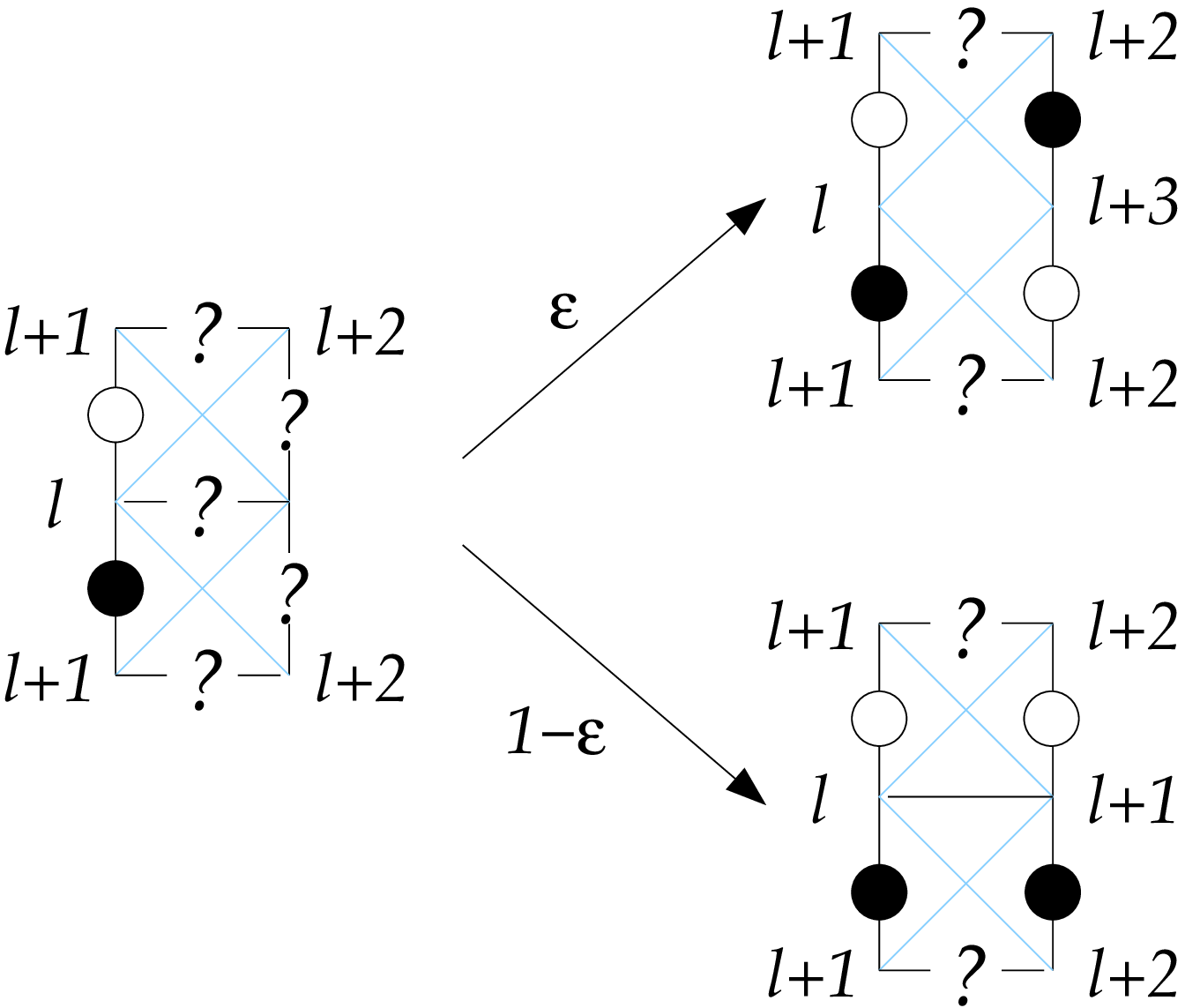}
\end{center}
\vspace{5mm}

Now, $D^{\Tore}(i+1,j)$ depends only on the edge $(i,j)\to (i+1,j)$: it is equal to $\ell+1$ if this edge is open and the particle lying at $j$ stays put.
If the edge $(i,j)\to (i+1,j)$ is closed, $D^{\Tore}(i+1,j)=\ell+3$, which corresponds for the particle lying at $j$ to a move to $j+1$.

We leave the cases in which a particle is followed by an other particle and in which an empty edge is followed by an empty edge, which are similar, to the reader.

We do the bottom boundary case (when a particle may enter the system, see the figure below). Assume that there is no particle at time $i$ 
in position $-K+1$. This means that if we set $\ell=D(i,-K)$ then $D(i,-K+1)=\ell+1$.
If the edge $(i,-K)\to(i+1,-K)$ is open then $D(i+1,-K)=\ell +1$ and there is no particle at time $i+1$ in position $-K+1$. 
Otherwise $D(i+1,-K)=\ell +3$ and a particle appears at time $i+1$ in $-K+1$.

\vspace{5mm}
\begin{center}
\includegraphics[width=65mm]{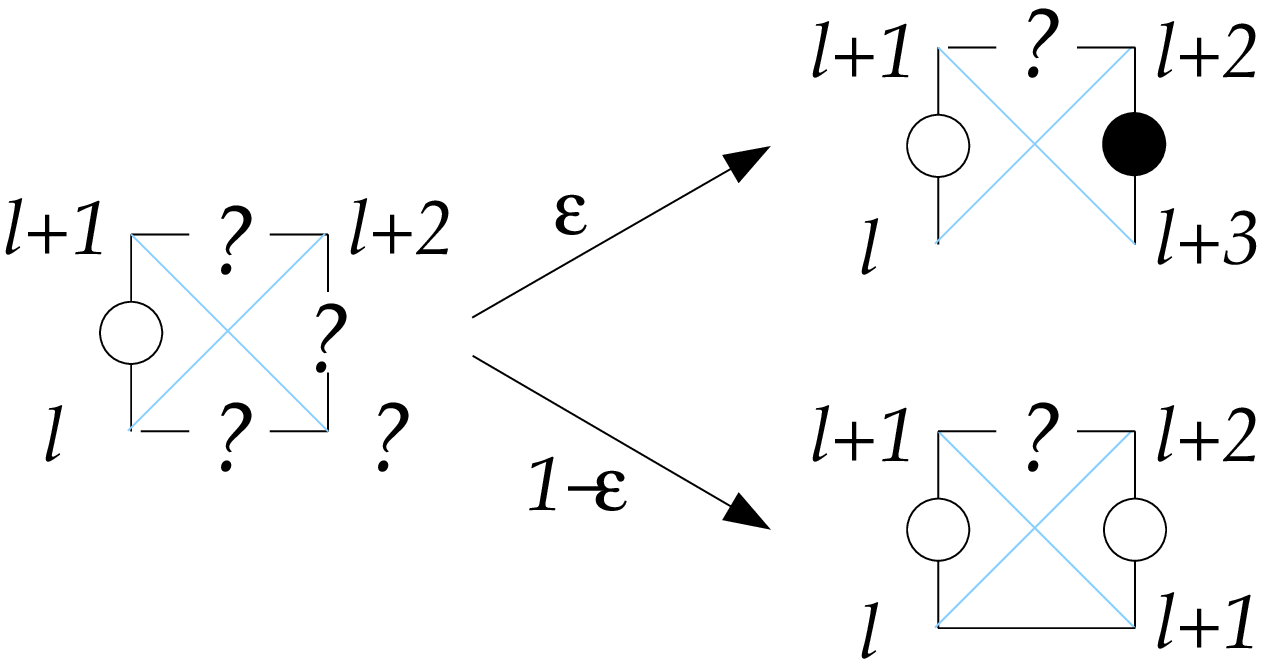}
\end{center}
\vspace{5mm}
The right-boundary case (when a particle exits) is similar.

\end{proof}
We have thus seen in the proof that the knowledge of $(\mathbf{Y}_i)_{i\geq0}$ fully determines the metric $D^{\Tore}$ using the following recursive identity:
$$D^{\Tore}(i+1,j)=D^{\Tore}(i,j)+1+2.\un_{\set{Y_i^j=\bullet, Y_i^{j+1}=\circ, Y_{i+1}^j=\circ, Y_{i+1}^{j+1}=\bullet }}.$$
Let $\nu_{K,\varepsilon}$ be the stationary measure of the synchronous TASEP $(\mathbf{Y}_i)_{i\ge 0}$.
Let
\begin{equation*}
\nu_{K,\varepsilon}(\bullet,\circ)\defeq\nu_{K,\varepsilon}(y^{0}=\bullet,y^{1}=\circ).
\end{equation*}

\begin{prop}\label{Coro:ConstanteTore}
We have the following asymptotics for the distances on $\Bande$ in the Cross Model:
\begin{equation*}
\lim_{n\rightarrow\infty}\frac{1}{n}\E(D^{\Tore}(n,0))= 1+2\varepsilon\nu_{K,\varepsilon}(\bullet,\circ).
\end{equation*}
\end{prop}

\begin{proof}[Proof of Proposition \ref{Coro:ConstanteTore}]
We have seen in the proof of Proposition  \ref{Prop:ChaineMarkov} that
\begin{equation*}
D^{\Tore}(i+1,0)=D^{\Tore}(i,0)+1,
\end{equation*}
unless a particle has moved at time $i$ from position $0$ to $1$, in which case $D^{\Tore}(i+1,0)=D^{\Tore}(i,0)+3$. This shows that
\begin{equation*}
\E(D^{\Tore}(i+1,0))=\E(D^{\Tore}(i,0))+1+2\P\set{Y^{0}_i=\bullet,Y^{1}_i=\circ, \{\mbox{$(i,0)\to(i+1,0)$ is closed}\}},
\end{equation*}
thus
\begin{equation*}
\E(D^{\Tore}(n,0))=n+2\varepsilon\sum_{j=0}^{n-1}\P\set{Y^{0}_j=\bullet,Y^{1}_j=\circ},
\end{equation*}
which gives, by the Markov chain ergodic theorem, the proof of the proposition.
\end{proof}

We thus need an estimate of $\nu_{K,\varepsilon}(\bullet,\circ)$. It turns out that the stationnary measure of the synchronous TASEP was studied in detail by Evans, Rajewsky and Speer \cite{ERS} using a matrix ansatz.

\begin{prop}\label{Prop:UnQuart}
The following identities relative to $\nu_{K,\varepsilon}(\bullet,\circ)$ hold:
\begin{itemize}
\item[\emph{(i)}] For all $K,\eps$,
\begin{equation*}
 \nu_{K,\varepsilon}(\bullet,\circ)= {A_\varepsilon(K)\over \varepsilon A_\varepsilon(K)+A_\varepsilon(K+1)},
\end{equation*}
with $A_\varepsilon(K)={1\over K}\sum_{k=1}^K\binom{K}{k}\binom{K}{k+1}(1-\varepsilon)^k$.
\item[\emph{(ii)}] \begin{equation*}
\lim_{K\rightarrow \infty} \nu_{K,\varepsilon}(\bullet,\circ)=\frac{1-\sqrt{1-\varepsilon}}{2\varepsilon}.
\end{equation*}
\item[\emph{(iii)}] For all $\alpha>0$,
$$
\lim_{\varepsilon\rightarrow 0} \nu_{\varepsilon^{-\alpha},\varepsilon}(\bullet,\circ)=\frac{1}{4}.
$$
\end{itemize}
\end{prop}

\begin{proof}
The first two assertions are consequences of (4.24), (8.21) and (10.13) of \cite{ERS}.
For (iii), we are led to prove that $\lim_{\varepsilon\rightarrow 0} {A_\varepsilon(\varepsilon^{-\alpha}+1)\over A_\varepsilon(\epsilon^{-\alpha})}=4$.

Denote $$a(K,k)\defeq\binom{K}{k}\binom{K}{k+1}(1-\varepsilon)^k.$$
The ratio
\begin{equation}{a(K,k+1)\over a(K,k)}=(1-\varepsilon){(K-k)(K-k-1)\over (k+1)(k+2)}\label{Ratio:a}\end{equation}
 is asymptotically equal to $(1-\varepsilon)({K-k\over k})^2$ for large values of $k$ and $K-k$.
 Therefore the sequence  $(a(K,k))_{1\leq k\leq K}$ increases from 1 to some $k_{max}(K)$ and decreases after.
Moreover, $k_{max}(K)\sim K({1\over2}-{\varepsilon\over8})$.

For $K>0$ and $\beta\in(0,1)$, let us decompose
the sum $K A_\varepsilon(K)$ into
$$\sum_{k=1}^{K({1\over2}-\varepsilon^\beta)-1}a(K,k)+\sum_{k=K({1\over2}-\varepsilon^\beta)}^{K\over 2} a(K,k)+\sum_{k={K\over2}+1}^{K({1\over2}+\varepsilon^\beta)}a(K,k)+\sum_{k=K({1\over2}+\varepsilon^\beta)+1}^Ka(K,k)$$
and denote by $A^-_{\beta}(K)$, $B^-_{\beta}(K)$, $B^+_{\beta}(K)$, $A^+_{\beta}(K)$, the four successive sums.
Note first, from the analysis of the ratio \eqref{Ratio:a},  that $A^-_{\beta}(K)$ and $A^+_{\beta}(K)$ are both subgeometric sums with rate $1- C\varepsilon^\beta$. This implies,
$$A^+_{\beta}(K)\leq C \varepsilon^{-\beta} a(K, K({1\over2}+\varepsilon^\beta))$$
$$A^-_{\beta}(K)\leq C \varepsilon^{-\beta} a(K, K({1\over2}-\varepsilon^\beta)).$$
In the same time, from the variations of the sequence $(a(K,k))_{1\leq k\leq K}$, we deduce
$$B^+_{\beta}(K)\geq K \varepsilon^\beta a(K, K({1\over2}+\varepsilon^\beta))$$
$$B^-_{\beta}(K)\geq K \varepsilon^\beta a(K, K({1\over2}-\varepsilon^\beta)).$$
 Therefore, as soon as the $K\gg \eps^{-2\beta}$, the ratio ${K A_\varepsilon(K)\over B^-_{\beta}(K)+B^+_{\beta}(K)}$ goes to 1.

 Now, for all $k$,
 $$ {a(K+1,k)\over a(K,k)}={(K+1)^2\over(K-k)(K-k+1)}.$$
 This ratio is between $4(1-c\varepsilon^\beta)$ and $4(1+C\varepsilon^\beta)$ when $k$ is between $K({1\over2}-\varepsilon^\beta)$ and $K({1\over2}+\varepsilon^\beta)$.
  This implies that, for all $\beta>0$,   the ratio ${ B^-_{\beta}(K+1)+B^+_{\beta}(K+1)\over B^-_{\beta}(K)+B^+_{\beta}(K)}$ converges to 4 as $\eps$ goes to 0 and $K$ goes to $\infty$.
Consequently, the same occurs also for ${(K+1) A_\varepsilon(K+1)\over K A_\varepsilon(K)}$, as soon as $K\gg \eps^{-2\beta}$.
We choose $\beta=\min\set{1, {\alpha\over4} }$ to prove the result.
\end{proof}

We will need further the following bounds on $\E(D^{\Tore}(n,0))$.
\begin{prop}\label{vitesseconv}
For $n\ge 0$, we have
\begin{equation*}
n(1+2\varepsilon\nu_{K,\varepsilon}(\bullet,\circ)) \le \E(D^{\Tore}(n,0))\le n(1+2\varepsilon\nu_{K,\varepsilon}(\bullet,\circ))+2K.
\end{equation*}
\end{prop}

\begin{proof}
By subadditivity of $(\E(D^{\Tore}(n,0)))_{n\ge 0}$, the sequence $\E(D^{\Tore}(n,0))/n$ is decreasing. Together with Proposition
\ref{Coro:ConstanteTore}, this proves the left inequality.

For the right inequality, the idea is to start the Markov chain $(\mathbf{Y}_i)_{i\ge 0}$ from its stationary distribution.
Let $\tilde{\mathbf{Y}}_0=(\tilde{Y}_0^j,j\in\IntG -K+1,K\IntD)$ with law $\nu_{K,\varepsilon}$. Take $\tilde{D}^{\Tore}(0,0)=0$ and define inductively $\tilde{D}^{\Tore}(0,j)$ for each $j$ such that
\begin{equation*}\tilde{D}^{\Tore}(0,j) =
\begin{cases}
&\tilde D^{\Tore}(0,j-1)-1  \text{ if } \tilde{Y}_0^j=\bullet.\\
&\tilde D^{\Tore}(0,j-1)+1  \text{ if } \tilde{Y}_0^j=\circ.\\
\end{cases}
\end{equation*}
If $(\tilde{\mathbf{D}}^{\Tore}_i)$ is a realization of the chain starting from $\tilde{D}^{\Tore}_0$ then, by Proposition \ref{Prop:ChaineMarkov}, we have
\begin{equation*}
\E(\tilde{D}^{\Tore}(n,0))= n(1+2\varepsilon\nu_{K,\varepsilon}(\bullet,\circ)).
\end{equation*}
By construction of the chain, there is a (random) $J$ such that
\begin{equation*}
\tilde{D}^{\Tore}(n,0)= D^{\Tore}\left((0,J)\to (n,0)\right)+ \tilde{D}^{\Tore}(0,J),
\end{equation*}
where $D^{\Tore}$ is, as before, the true distance after percolation in $\Bande$.
Using triangular inequality, we get
\begin{eqnarray*}
D^{\Tore}(n,0)&\le &K+D^{\Tore}\left((0,J)\to (n,0)\right)\\
&\le & K+\tilde{D}^{\Tore}(n,0)-\tilde{D}^{\Tore}(0,J)\\
&\le & 2K+\tilde{D}^{\Tore}(n,0).
\end{eqnarray*}
Taking expectation gives the expected bound.
\end{proof}

If we gather the results obtained so far in this section, we roughly obtain that for the Cross Model,
\begin{equation*}
\E(D^{K,d}(n,0)) \approx n(1+2\varepsilon\nu_{K,\varepsilon}(\bullet,\circ)) \approx n\left(1+\frac{\eps}{2}\right).
\end{equation*}
In order to apply this result to usual percolation on strips (with horizontal and vertical edges open with probability $1-\eps$), we have to prove that distances in both models differ very little. Fortunately, it happens that distances in both models are quite similar as soon as there are no contiguous closed edges, which is the case with high probability in any fixed rectangle, when $\eps$ is small.

We first choose in the Cross Model a particular path among all the minimal paths:

\begin{lem}\label{Rem:ModificationChemin}
For the Cross Model, for each $K,n$, there exists a path from $(0,0)$ to $(n,0)$ of minimal length that only goes through horizontal and diagonal edges, except possibly through the vertical edges of the first column $\set{0}\times\IntG -K,K\IntD$.
\end{lem}

\begin{proof}
Consider one of the minimal paths from $(0,0)$ to $(n,0)$, we will modify this path $\mathcal{P}$ to get one which satisfies the desired property. Among all the vertical edges of $\mathcal{P}$, let $\mathcal{E}=(i,j-1)\to (i,j)$ be the one for which
\begin{enumerate}
\item $i$ is maximal,
\item among them, $j$ is minimal.
\end{enumerate}
In this way, there is no vertical edges neither to the right of $\mathcal{E}$ nor below it. There are three cases according to whether the other edge joining $(i,j-1)$ in $\mathcal{P}$ goes to $(i-1,j-1)$, to $(i-1,j)$ or to $(i-1,j-2)$. In each of these three cases we can do a substitution that removes the vertical edge or moves it  to the left:
\begin{center}
\includegraphics[width=130mm]{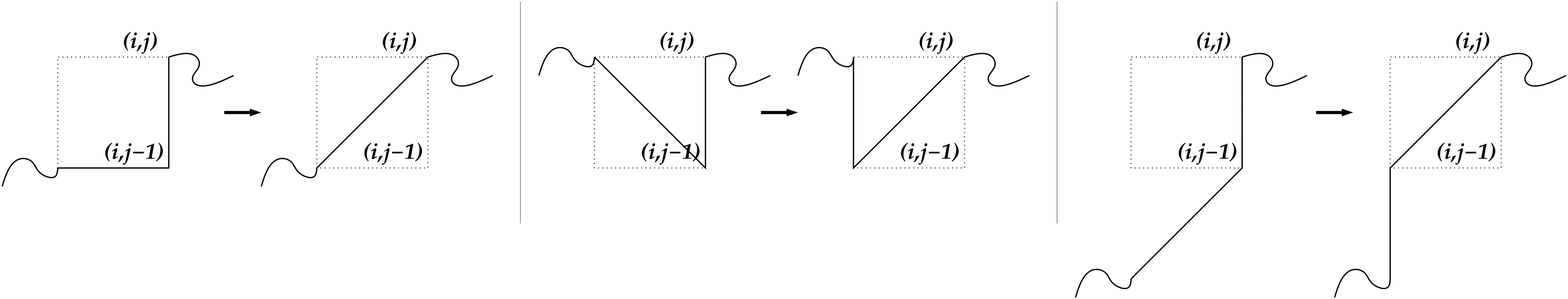}
\end{center}

This substitution does not change the length of the path, it is thus still optimal. After iterating the process, there may remain vertical edges 
only on the first column.
\end{proof}


\begin{prop}\label{Prop:EchelleConditionnee} Consider a standard percolation on $\Bande$ where each horizontal and vertical edge is open with probability $1-\eps$. Denote $D^K$ the associated distance.
Let $A=A(K,n,\eps)$ be the event "in each square of area 1 of $\IntG 0;n\IntD\times \IntG -K;K\IntD$ \emph{i.e.} of vertices $\{(i,j),(i,j+1), (i+1,j+1), (i,j+1)\}$, at most one edge is closed". Then, for each $n,K$,
\begin{equation*}
\E(D^K(n,0)\;|\;A )\leq \E(D^{\Tore}(n,0))+3K.
\end{equation*}
\end{prop}
\begin{proof}
On the event $A$, adding diagonal edges with length 2 does not decrease the length of optimal paths since either the path $(i,j)\rightarrow (i,j+1) \rightarrow (i+1,j+1)$  or $(i,j)\rightarrow (i+1,j) \rightarrow (i+1,j+1)$ is open and moreover, there is at most a distance 3 between the two extremities of any edge of $\IntG 0;n\IntD\times \IntG -K;K\IntD$.

Thanks to Lemma \ref{Rem:ModificationChemin}, there is always for the Cross Model a path of minimal length using only horizontal and diagonal edges except maybe on the at most $K$ vertical steps along the first column $\set{0}\times\IntG -K,K\IntD$. Thus, on the event $A$, we get
\begin{equation*}
D^K(n,0) \le D^{K,d}(n,0)+3K.
\end{equation*}
This yields
\begin{equation*}
\E(D^K(n,0)\;|\;A )\le\E(D^{K,d}(n,0)\;|\;A)+3K.
\end{equation*}
Let us note now that $A$ is an increasing event and $-D^{K,d}(n,0)$ is an increasing random variable.
Thus, the FKG inequality (see \cite{Grimmett}) yields
\begin{equation*}
\E(-D^{K,d}(n,0)\mathbf{1}_A )\ge \E(-D^{K,d}(n,0))\P\{A\}
\end{equation*}
which can be rewritten
\begin{equation*}
\E(D^{K,d}(n,0)\;|\;A )\le \E(D^{K,d}(n,0)).
\end{equation*}
\end{proof}

\section{The lower bound}
We now return to the original model and consider percolation on $\Z^2$ where each edge is closed independently
with probability $\varepsilon=1-p$. Recall that $D(i,j)$ denotes the distance between the origin and $(i,j)\in \Z^2$.
Using Proposition \ref{Prop:1ereBorne}, in order to prove the lower bound in Theorem \ref{main}, we just need to show that
\begin{equation*}\lim_{k\rightarrow \infty} \frac{1}{k} \E(D(T_1(k),0)\un_{\set{\Zero\leftrightarrow \infty}})\ge 1+\frac{\varepsilon}{2} + o(\varepsilon).
\end{equation*}
Let $D^d(x)$ be the distance between the origin and $x$ in $\Z^2$, with diagonal and vertical edges all open.
Adding edges decreases the distances: $D^d(x)$ is always smaller than $D(x)$.
Thus, we have,
\begin{eqnarray*}
\E(D(T_1(k),0)\un_{\{\Zero\leftrightarrow \infty\}})&\ge& \E(D^{d}(T_1(k),0)\un_{\{\Zero\leftrightarrow \infty\}})\\
&\ge&\E(D^{d}(k,0)\un_{\{\Zero\leftrightarrow \infty\}} ),
\end{eqnarray*}
since $T_1(k)\geq k$ and the sequence $(D^d(n,0))$ is increasing (thanks to vertical edges).

Note that every minimal path from the origin to $(k,0)$ (with diagonal and vertical edges open) has a length less than $2k$ and lies
in the strip $\Z_k$. Therefore
$$
\E(D^{d}(k,0)\un_{\{\Zero\leftrightarrow \infty\}})= \E(D^{k,d}(k,0)\un_{\{\Zero\leftrightarrow \infty\}}).
$$
Using Proposition \ref{vitesseconv}, we have
$\E(D^{k,d}(k,0))\ge k(1+2\varepsilon\nu_{k,\varepsilon}(\bullet,\circ))$.
Besides,
\begin{equation*}
\E(D^{k,d}(k,0)(1-\un_{\{\Zero\leftrightarrow \infty\}}))\le 2k\P\{0\nleftrightarrow \infty\}\le 4k\eps^4.
\end{equation*}
We get
\begin{equation*}\E(D(T_1(k),0)\un_{\{\Zero\leftrightarrow \infty\}})\ge k(1+2\varepsilon\nu_{k,\varepsilon}(\bullet,\circ)-4\eps^4).\end{equation*}
Letting $k$ tends to infinity, we obtain with Proposition \ref{Prop:UnQuart} the desired lower bound.

\section{The upper bound: a short path}

As in the previous section, we consider standard percolation on $\Z^2$ where each edge is closed with probability $\varepsilon=1-p$.
Recall that Proposition \ref{bornemu} states that for any $\delta>1$ there exists a $c_\delta>0$
such that, for all $\eps$ small enough and $n\ge 10$,
\begin{equation}\label{eqbornemu}
\mu_{p}\le \frac{\E(D(n,0)\infin)}{n}+c_\delta\varepsilon^2n^{\delta-1}.
\end{equation}
So to prove Theorem \ref{main}, it is sufficient to find  $n=n(\varepsilon)\ge 10$ such that $\varepsilon n^{\delta-1}=o(1)$ and
\begin{equation*}\E(D(n,0)\infin)\le n(1+\frac{\varepsilon}{2} + o(\varepsilon)).\end{equation*}
Fix $\ell>1$ and $K,n> 2\ell$ and   define the boxes $(C_i)_{i \ge 0}$ by
\begin{equation*}C_i=[\ell,n-\ell]\times[(2i-1)K+1,(2i+1)K].\end{equation*}
Let $E_i$ denote the event
"in each square of area 1 of $C_i$ \emph{i.e.} of vertices $\{(i,j),(i,j+1), (i+1,j+1), (i,j+1)\}$, at most one edge is closed".
Let us note that the events $(E_i)_{i\ge 0}$ are independent and have the same probability. Moreover, we have
\begin{align*}
\P\set{E_i^c}&\le |C_i|\P\{\mbox{more than $2$ edges are closed in a given square}\}\\
&\le 2Kn(6\varepsilon^2+4\eps^3+\eps^4)\le 22Kn\eps^2.
\end{align*}
Hence, if $I=\inf\{i\ge 0, E_i\}$, $I+1$ is geometrically distributed  with parameter $\P\{E_0\}$.

\begin{figure}[h!]
\label{Fig:BoitesEiAvecChemins}
\begin{center}
\includegraphics[width=110mm]{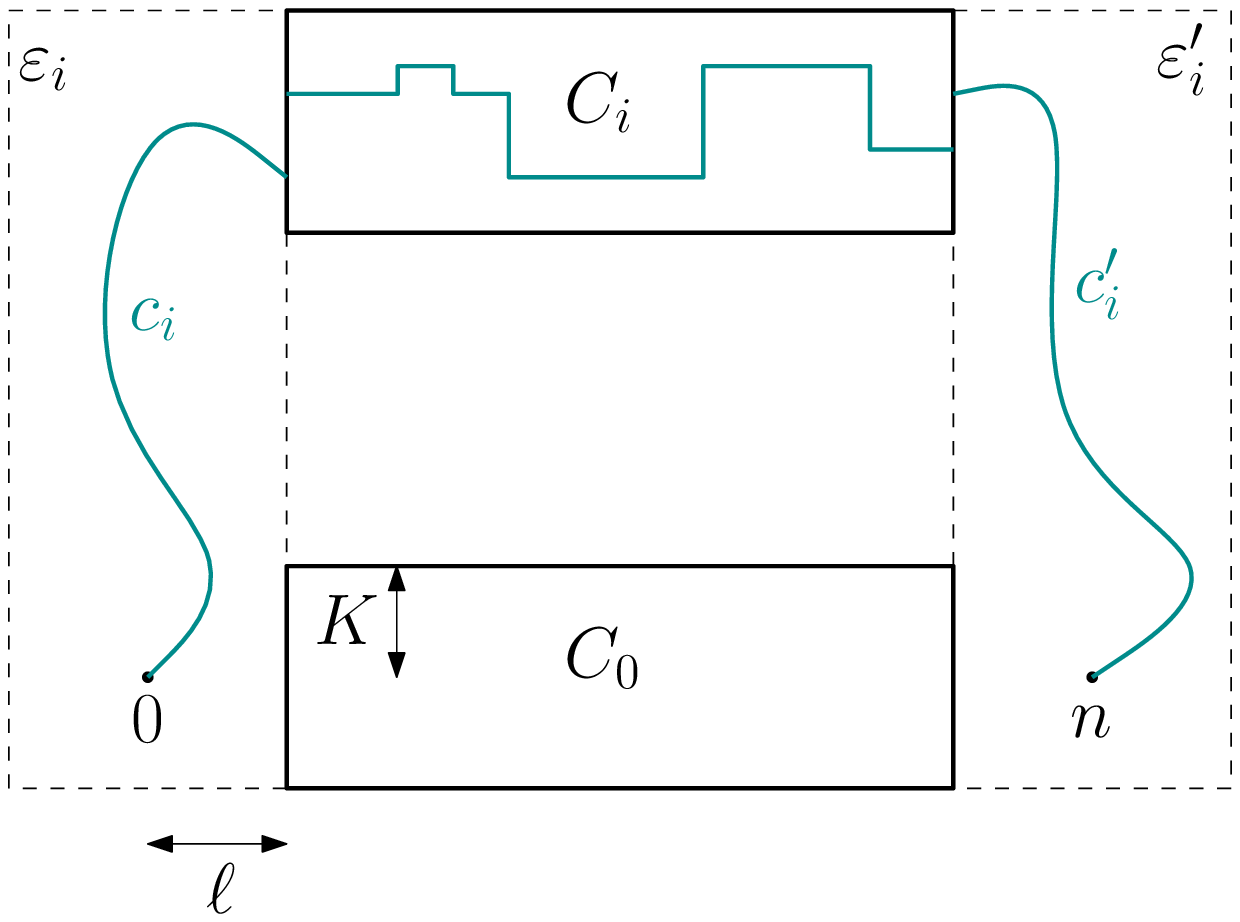}
\caption{The sketch of an almost optimal path, on the event $B_i\cap \set{I=i}$.}
\end{center}
\end{figure}

Let $c_i, c'_i$ denote the vertical boundaries of $C_i$ \emph{i.e.} $c_i\defeq\{\ell\}\times[(2i-1)K+1, (2i+1)K]$ and $c'_i\defeq\{n-\ell\}\times[(2i-1)K+1, (2i+1)K]$
and $\varepsilon_i,\varepsilon'_i$ the boxes defined by
\begin{align*}
\varepsilon_i &=[-\ell,\ell]\times [-K , (2i+1)K],\\
\varepsilon'_i &=[n-\ell,n+\ell]\times [-K , (2i+1)K].
\end{align*}
Let $B_i$ be the event
\begin{equation*}
B_i=\{\Zero\leftrightarrow c_i \mbox{ in } \varepsilon_i\}\cap \{(n,0)\leftrightarrow c'_i \mbox{ in } \varepsilon'_i\}.
\end{equation*}
We have
\begin{equation}\label{eq1}
\E(D(n,0)\infin)\le \sum_{i=0}^\infty \E(\un_{\{I=i\}\cap B_i}D(n,0))+\sum_{i=0}^\infty \E(D(n,0)\un_{\{I=i\}\cap B_i^c \cap \{\Zero\leftrightarrow (n,0)\leftrightarrow \infty\}}).
\end{equation}

Let $\chi_1(i)$ (resp. $\chi'_1(i)$) denote the distance between $\Zero$ and the segment $c_i$ inside the box $\varepsilon_i$ (resp. $(n,0)$ and the segment $c'_i$ inside the box $\varepsilon'_i$). Denote also $\chi_2(i)$ the distance between  $(\ell,2iK)$ and $(n-\ell,2iK)$ inside the box $C_i$.
We have
\begin{equation}\label{eqdecomp}
D(n,0)\le \chi_1(i)+\chi'_1(i)+\chi_2(i)+\max_{x\in c_i} D(x\to(\ell,2iK))+\max_{y\in c'_i} D(y \to(n-\ell,2iK)).
\end{equation}
Moreover, on the event  $\{I=i\}\cap B_i$, the r.h.s. of this inequality is finite and we have
\begin{equation}\label{eqchi1}
\chi_1(i)\le |\varepsilon_i|=4\ell(i+1)K \qquad \chi'_1(i)\le |\varepsilon'_i|=4\ell(i+1)K
\end{equation}
\begin{equation}\label{eqci}
\max_{x\in c_i}D(x\to(\ell,2iK))\le 3K \qquad \max_{y\in c'_i}D(y\to(n-\ell,2iK))\le 3K.
\end{equation}
To bound $\chi_2(i)$, let us note that $\chi_2(i)$ only depends on the edges inside the box $C_i$. Thus, we have
\begin{eqnarray*}
\E(\chi_2(i)\un_{\{I=i\}\cap B_i}) &\le& \E(\chi_2(i)\un_{\{I=i\}})\\
& = & \E(\chi_2(i)\;|\; {I=i})\P\{I=i\}\\
& = & \E(\chi_2(0)\;|\; {I=0})\P\{I=i\}.\\
\end{eqnarray*}
Moreover, the event $\{I=0\}$ coincides with the event $A(K,n-2\ell,\varepsilon)$ defined in Proposition \ref{Prop:EchelleConditionnee}.
This yields, using Proposition \ref{vitesseconv} and Proposition \ref{Prop:EchelleConditionnee},
\begin{equation}\label{eqchi2}
\E(\chi_2(0)\;|\; {I=0}) \le \E(D^{\Tore}(n-2\ell,0))+3K\le n(1+2\varepsilon\nu_{K,\varepsilon}(\bullet,\circ))+5K.
\end{equation}
Combining (\ref{eqdecomp}),(\ref{eqchi1}),(\ref{eqci}),(\ref{eqchi2}), we get
\begin{equation}\label{eqrhs1}
\sum_{i=0}^\infty \E(\un_{\{I=i\}\cap B_i}D(n,0))\le
 n(1+2\varepsilon\nu_{K,\varepsilon}(\bullet,\circ))+5K+ K(8\ell\E(I+1)+6).
 \end{equation}
For the second term in the right hand side of \eqref{eq1}, using H\"older inequality, we get
\begin{multline}\label{eqcheminlong}
\E(D(n,0)\un_{\{I=i\}\cap B^c_i\cap \{0\leftrightarrow (n,0)\leftrightarrow \infty\}})\\ \le \P\{B_i^c \cap \{I=i\} \cap \{\Zero\leftrightarrow (n,0)\leftrightarrow \infty\}\}^{1/2}\E(D^2(n,0)\infin)^{1/2}.
\end{multline}
Lemma \ref{roughbound} yields
$$\E(D^2(n,0)\infin)^{1/2}\le C_{\delta}n^{\delta}.$$
Besides, we have
\begin{eqnarray*}
\P\{B_i^c \cap \{I=i\} \cap \{\Zero\leftrightarrow (n,0)\leftrightarrow \infty\}\}&\le& 2\P\{\{0\nleftrightarrow c_i \mbox{ in } \varepsilon_i\}\cap\{0\leftrightarrow \partial \varepsilon_i\}\cap \{I=i\}\} \\
& = & 2\P\{\{0\nleftrightarrow c_i \mbox{ in } \varepsilon_i\}\cap\{0\leftrightarrow \partial \varepsilon_i\}\}\P \{I=i\},
\end{eqnarray*}
since the two first events only depend on the edges inside the box $\varepsilon_i$ and the event $\{I=i\}$ only depends on the edges inside the box $C_i$. If $\Zero\leftrightarrow \partial \varepsilon_i$ but is not connected to $c_i$ in the box $\eps_i$, then there exists in the dual graph a path of closed edges with at least $2\ell$ edges. Hence
\begin{equation*}
\P\{\{0\nleftrightarrow c_i \mbox{ in } \varepsilon_i\}\cap\{0\leftrightarrow \partial \varepsilon_i\}\}\le 2K(i+1)2\ell(3\varepsilon)^{2\ell}.
\end{equation*}
Plugging this into \eqref{eqcheminlong} gives
\begin{equation}\label{eqrhs2}
\sum_{i=0}^\infty \E(\un_{\{I=i\}}\un_{B_i^c}D(n,0)\infin)
\le (8\ell K)^{1/2}(3\varepsilon)^{\ell}C_\delta n^{\delta}\sum_{i=0}^\infty \P\{I=i\}^{1/2}(i+1)^{1/2}.
\end{equation}
Recall now that $I+1$ is a geometric random variable with parameter $\P\{E_0\}\ge 1-22Kn\varepsilon^2$. Thus, there exists $C<\infty$ such that, for any $K,n,\eps$ are such that
$22Kn\varepsilon^2\le \frac{1}{2}$, we have
\begin{equation*}\sum_{i=0}^\infty \P\{I=i\}^{1/2}(i+1)^{1/2}\le C \quad\mbox{ and } \quad\E(I)\le C.
\end{equation*}
Combining \eqref{eqbornemu}, \eqref{eq1}, \eqref{eqrhs1} and \eqref{eqrhs2}, we get, for $\varepsilon \in (0,\frac{1}{6})$, $K,n\ge 2\ell$ and such that $22Kn\varepsilon^2\le \frac{1}{2}$
\begin{equation*}
\mu_{p}\le (1+2\varepsilon\nu_{K,\varepsilon}(\bullet,\circ))+ C_{\delta,\ell}(n^{\delta-1}\eps^{\ell}K^{1/2}+\frac{K}{n}+\eps^2n^{\delta-1})
\end{equation*}
where $C_{\delta,\ell}$ is a constant depending only on $\ell$ and $\delta>1$.
Take now $\ell=3$, $n=\varepsilon^{-3/2}$, $K=\varepsilon^{-1/4}$ and $\delta=3/2$. We get
\begin{equation*}
\mu_{p}\le 1+2\varepsilon\nu_{\varepsilon^{-1/4},\varepsilon}(\bullet,\circ)+C\varepsilon^{5/4}.
\end{equation*}
We conclude by using Proposition \ref{Prop:UnQuart}.

\bibliographystyle{imsart-nameyear}

\end{document}